\documentclass[a4paper,10pt]{article}

\usepackage{times}
\usepackage[T1]{fontenc}
\usepackage[latin1]{inputenc}
\usepackage{amsmath,amssymb}
\usepackage{amsthm}
\usepackage{amscd}

\usepackage{graphicx} 

\usepackage{epstopdf}
\newtheorem{theorem}{Theorem}[section]

\newtheorem{proposition}[theorem]{Proposition}

\theoremstyle{definition}
\newtheorem{example}[theorem]{Example}
\newtheorem{remark}[theorem]{Remark}

\numberwithin{table}{section}
\numberwithin{equation}{section}

\title{Multicentric representation of piecewise constant holomorphic functions and Hermite interpolation }
\author{Olavi Nevanlinna, Tiina Vesanen}
\date{10.11.2025}

\begin{document}

\maketitle

 \begin{center}
{\footnotesize\em 
Aalto University\\
Department of Mathematics and Systems Analysis\\[3pt]
 email: Olavi.Nevanlinna\symbol{'100}aalto.fi\\ 
Tiina.Vesanen\symbol{'100}aalto.fi\\[3pt]
}
\end{center}

\begin{abstract}
 
\end{abstract}  
\bigskip
In multicentric representation of piecewise holomorphic functions one combines Lagrange interpolation  at  roots of a polynomial $p$  with  convergent power series  of $p$  as the "coefficients"   multiplying the Lagrange basis polynomials.  When these power series are truncated one obtains Hermite interpolation polynomials.  In this paper we   first  review different approaches to obtain multicentric representations with emphasis in piecewise constant holomorphic functions. 

When the polynomial is of degree $d$ and all  power series are truncated after $n^{th}$ power,  we formally arrive into a Hermite interpolation polynomial of degree  $d  (n+1) -1 $.  
The natural way to represent  Hermite interpolation is to have  for each interpolation condition a basis polynomial which  in this case leads to $d(n+1)$ basis polynomials.     We then consider the  numerical accumulation of errors in the different ways to represent and evaluate the Hermite interpolation.  In the multicentric representation  due to the convergence of the power series,  numerical errors stay bounded as $n$  grows.   When we assume  that the  piecewise constant holomorphic function  takes the value $1$ in one of the components and vanishes  in the other so that the Hermite interpolation  agrees with just one basis  polynomial, even then the  truncated multicentric representation is  favorable.  In the general  case one  would take a linear combination of  all $d(n+1)$  basis polynomials.

\bigskip


\smallskip


\newpage

\section{Introduction}

\bigskip
  
Using {\it multicentric calculus} \cite{Ne12a} it is possible to get stable and convergent power series based representations for piecewise constant holomorphic functions.  When these series expansions are truncated,  we obtain polynomials which formally agree with  polynomials obtained as Hermite interpolants.  In this paper we  present  different ways of creating the multicentric representations.  Then we comment different  ways of creating Hermite interpolants  for the same purpose.  All these different representations can be done with exact arithmetics, but we shall then compare their performance under assumptions where  we allow in evaluations of variables tiny errors.  In the representations obtained from the multicentric form the errors do not accumulate essentially at all, unlike in the other expressions.

While with Hermite interpolation we may form a basis from polynomials,  of which the interpolative  polynomial can be obtained as linear combination, in multicentric representation we use a fixed Lagrangian interpolation basis  but   consider the coefficients as functions. And it is the series expansions of these "coefficients" which   when truncated yield  formally the Hermite interpolant.  This  is the cause for a very different order of steps in  evaluating the polynomials and is visible in  the sensitivity in error propagation. 

\bigskip
Different approaches to get the multicentric representations are surveyed in Section 2 but the point of departure is in {\it taking  a polynomial as a new variable}.    We assume that a polynomial $p$ of degree $d$ with simple roots $\lambda_k$ has been selected and  the lemniscate set
$$
V(\rho) =\{ z \in \mathbb C: \ |p(z)| \le \rho \} 
$$
has several components.  In {\it multicentric calculus} we represent  scalar  functions 
$$
\varphi:  V(\rho) \rightarrow \mathbb C, \ \  z \mapsto \varphi(z)
$$
by vector valued functions
$$
f: B(\rho) \rightarrow \mathbb C^d,   \ \   w \mapsto f(w) 
$$
so that  if $z \in V(\rho)$ then the new variable $w= p(z)$ is in the disc $B(\rho)$ and within the lemniscate 
\[
\gamma_\rho = \{ z \in \mathbb C: \ |p(z)| = \rho \}.
\]   
The original function can then be recovered by
\begin{equation}\label{multicentric}
\varphi(z) = \sum_{k=1}^d \delta_k(z) f_k(p(z))
\end{equation}
where the polynomial $\delta_k$ is, as in the Lagrange interpolation, the polynomial of degree $d-1$ such that $\delta_k(\lambda_k)=1$, while $\delta_k(\lambda_j)=0$ for $j\not= k$.

 Jacobi \cite{Jac} considered (without the Cauchy integral) expansions of the form
 $$
 \varphi(z) \sim \sum_{j=0}^\infty c_j(z) p(z)^j
 $$
 where $c_j$ are polynomials of degree less than $d$,  see e.g.  \cite{Ki1906, Fe32, Wa69, Ma77}.   However, in this tradition  $w=p(z)$ was not 
 considered a truly new variable.
 Hence, the real potential of  utilising power series in discs for computational purposes was overlooked.  We call this multicentric calculus;  observe that $|p(z)|^{1/d}$ is the geometric average of distances from $z$ to the "local centers" $\lambda_k$.     On this development see \cite{Ne12a, Ne12b, Ap16, Ne16, An19},  but we aim here to be sufficiently  self-contained.  
 
 An important application of multicentric calculus  is in holomorphic functional calculus, where one needs an accurate and and easily computable approximation of  holomorphic functions $\varphi$ near the spectrum $\sigma(A)$ of a given bounded operator $A$.  When the spectrum  lies inside the components bounded by the lemniscate, then the spectrum of $B= p(A)$ lies inside the corresponding disc and if we have power series 
 $$ f_k(w)  = \sum_{j=0} ^\infty \alpha_{k,j} w^j
 $$
 available and converging fast, then only a moderate number of terms of the power series is needed to represent   $f_k(B)$, after which  an approximation of
 $$
 \varphi(A) = \sum_{k=1}^d \delta_k(A) f_k(B)
 $$
 can be computed.  Notice that if, say for a sparse large scale matrix $A$,  the  matrix $\varphi(A)$ is only needed at a vector $b$, then all  multiplications needed to create the approximate $\varphi(A)b$ can be obtained using matrix-vector multiplications only.

Here we concentrate on holomorphic functions which are piecewise constant. However, this can be done by further  narrowing the focus on {\it functions which are identically $1$ in one component and  vanish in the others}  as the general case can be obtained by superposing. 
In this approach we have approximation errors mainly  from the truncation of power series of the components of $f$  and numerical errors from  evaluation of the  truncated expansions at the  operator or matrix $B=p(A)$.  However,  it is important to notice that we  have the liberty to choose the local centers $\lambda_k$ to have rational real and imaginary parts as then  in  the multicentric representation the coefficients in the power series of  the components of $f$ are all rational, and the computations to create them can be done with exact arithmetic.  

\begin{proposition}
Assume that all roots $\lambda_k$  of $p$ have rational real and imaginary parts.  If  $\varphi$  is holomorphic at the roots  and all values $\varphi^{(\nu)}(\lambda_k)$ have also rational real and imaginary parts then the series expansions of $f_k$ have coefficients with rational real and imaginary parts only.
\end{proposition}
\begin{proof}
This follows immediately from the explicit recursion, see \cite{Ne12a, An19} or  (\ref{explicitrecursion}) below.  
\end{proof}

We  use the following notation for truncating power series.  If 
\begin{equation} \label{katkaisu}
a(x) \sim \sum_{j=0} ^\infty c_j x^j , \ \  \text { then }   \lfloor a \rfloor_n (x) = \sum_{j=0}^n c_j x^j.
\end{equation}
Thus, when we truncate  the series for $f_k(w)$ we can write at $w=p(z)$
\begin{equation}\label{katkaisunmerkitys}
\lfloor f_k \rfloor_n (p(z)) = \sum_{j=0}^n \alpha_{k,j}\  p(z)^j.
\end{equation}

\bigskip
Consider now the Hermite interpolation problem: at $d$ distinct  complex points $\lambda_k$  to find polynomial $P$ such that  
\begin{equation}\label{hermiteyleinen}
P^{(\nu)}(\lambda_k) = a_{k,\nu} \ \text{ for } \nu= 0, \dots, n; \ \  k=1,\dots,d.  
\end{equation}
\begin{proposition}
Let $\varphi$ be holomorphic at interpolation points $\lambda_k$. Then  the truncated polynomial of the multicentric representation  of $\varphi$
$$
P(z)= \sum_{k=1}^d \delta_k(z) \lfloor f_k\rfloor_n (p(z))
$$ is the unique polynomial of degree $N \le d(n+1)-1$ satisfying the interpolation problem (\ref{hermiteyleinen}) when $a_{k,\nu}= \varphi^{(\nu)}(\lambda_k)$.
\end{proposition}
 
\begin{proof}
Clearly, $P$ is of at most degree $d(n+1)-1$. Consider the difference $\varphi-P$. We have 
$$
\varphi(z) - P(z) = p(z)^{n+1} \; \sum_{k=1}^d \delta_k(z) \negthickspace \sum_{j=n+1}^\infty \! \alpha_{k,j} \, p(z)^{j-n-1}
$$
which shows that the interpolation conditions are all satisfied.
\end{proof}
 
\bigskip
 
Denote by $\delta_{k,j,n}$  the solution to the special case of (\ref{hermiteyleinen} )  with $a_{k,j}=1$, while $a_{m, \nu} =0$ when $(m,\nu) \not=(k,j)$.  Then $(\delta_{k,j,n})$  is a basis  such that the general problem (\ref{hermiteyleinen}) is solved by
$$
P(z) = \sum_{k=1}^d \sum_{j=0}^n a_{k,j} \  \delta_{k,j,n}(z).
$$

We shall keep writing $\varphi$ and $f$ when we discuss general  formulae but use $\chi$ and $g$  in the  particular case where $\chi(z) = 1$ in the component containing the root $\lambda_1$  and vanishing in the components containing  other roots, so that
$$
\chi(z) = \sum_{k=1}^d \delta_k(z) g_k(p(z)). 
$$
When we truncate   here the power series of $g_k$'s we  get the basis function $\delta_{1,0,n}$  which  has the multicentric representation
\begin{equation}\label{kantapolynom}
\delta_{1,0,n}(z) = \sum_{k=1}^d \delta_k(z) \lfloor g_k \rfloor_n (p(z)).
\end{equation}
This  representation of  the polynomial $\delta_{1,0,n}$  is later denoted by $M_n$, letter $M$ for multicentric.

In Section 3 we indicate different direct ways to evaluate this basis polynomial.  With $H_n$ we denote the direct representation  as powers of $z$,  $S_n$ shall denote a special representation of  the basis polynomial and $T_n$ is reserved for a particular formulation of the basis polynomial suitable for parallel computing.  

In Section 4  we  consider the numerical stability of these different ways and in particular see that the order of computational steps obtained from truncation the multicentric representation is superior, as we let $n$  grow  while keeping the interpolation points, the roots of $p$ fixed.  More specifically,   we shall demonstrate their  different sensitivity using the following  two polynomials, one with real zeros symmetrically around he origin 
  \begin{equation}\label{symmetric}
  p(z)= z (1-z^2)
  \end{equation}
  and  another, nonsymmetric one:
  \begin{equation}\label{nonsymmetric}
  p(z) = z( z^2+4z+5).
  \end{equation}
  In both polynomials we choose $\lambda_1=0$ to be the root  lying in the component where the piecewise holomorphic function would be identically 1.  Thus, the first one  models a situation where we would be using Riesz functional calculus to concentrate in the middle of the spectrum while the second could be thought as modelling a situation where one would like to identify the invariant subspace of an operator corresponding to the unstable part  of the spectrum. Notice that the behaviors are independent of translating and scaling of the  basic variable $z$.

 \section{Basic formulas  for multicentric representation}
 
 \bigskip
 In this section we present different approaches to create the multicentric representations for holomorphic functions.  
 Formulas are given for general holomorphic functions $\varphi$, and then $f$ is used for the vector valued representing function, while for piecewise constant functions we use $\chi$  with $g$ as the representing function, as the formulas  simplify considerable.  
 We  begin with splitting the Cauchy kernel and then obtain the general  integral representations for the components of $f$.  
 The integral representation can be used to  obtain effective bounds for the truncation errors and for the speed of the convergence.   Then we remind that the contour integrals can be obtained by residue calculus. 
 A completely different starting point is to use the polyproduct algebra  for which  the functional equation $\chi^2 = \chi$ allows one to make a power series ansatz for  $g$  and to solve the coefficients recursively from $g\circledcirc g=g$.  
 Finally, if we are given  $\varphi$ as a power series of the original variable $z$, then  we can obtain $f$ from knowing $Z_j^{\circledcirc k}$  which represents $z^k$ by  transforming the power series termwise. 
 
 \subsection{Cauchy integrals} \label{cauchysection}

Assume $\rho$ is such that the lemniscate $\gamma_\rho$ 
does not pass through any critical point of $p$. We  decompose the Cauchy kernel in order to  get integral representation for the components of the vector function $f(w)$,  see
 \cite{Ne12a}. 
 
We decompose the Cauchy  kernel $\frac{1}{\lambda - z}$   writing 
\begin{equation}\label{decompCauchy}
\frac{1}{\lambda - z} = \frac{1}{p(\lambda)- p(z)}  \ \sum_{k=1}^d  p'(\lambda_k) \delta_k(\lambda) \delta_k(z). 
\end{equation}
Denoting 
\begin{equation}\label{decompkernel}
K_k(\lambda, w) = \frac{1}{\lambda - \lambda_k} \frac{p(\lambda)}{p(\lambda)-w}
\end{equation} we obtain for $f_k$ the expression
\begin{equation}\label{fexpressed}
f_k(w) = \frac{1}{2\pi i } \int_ {\gamma_\rho} K_k(\lambda, w) \varphi(\lambda)d\lambda,
\end{equation}
where $\varphi$ is holomorphic inside $\gamma_\rho$ and continuous up to boundary.  Here we are interested in piecewise continuous $\varphi$ which means  that if the lemniscate would pass  through a critical point, then $\varphi$ would have to take  the same value in each component  touching this critical point.   Without loss on generality we assume that  $\rho$ is such  that the lemniscate   contains no critical points and in particular the lemniscate  consists of  smooth curves,  each surrounding at least one $\lambda_k$.  Then, if $|w| < |p(\lambda)| = \rho$,   the series
$$
f_k(w) = \sum_{n=0} ^\infty \alpha_{k, n} w^n
$$
converges where
\begin{equation}\label{expressderivatives}
\alpha_{k, n} \ = \ \frac{1}{2\pi i} \int_{\gamma_\rho} \frac{1}{ \lambda - \lambda_k } \  \frac{\varphi(\lambda)}  {p(\lambda)^n} \ d \lambda. 
\end{equation}

We shall assume that $\chi$ takes  only values 1 and 0, as by linearity  the other piece wise constant functions can be  expressed as linear combination of those.

\begin{example} \label{example1}
Let $p(z)= z (1-z^2)$  and  set $\lambda_1=0, \lambda_2=-1, \lambda_3=1$, so that 
\begin{equation}
\delta_1(z) = 1-z^2 ,\ \  \delta_2(z) = \frac{1}{2} (z^2-z)  \  \text { and }  \ \delta_3(z) = \frac{1}{2} (z^2+z).
\end{equation}
Consider the piecewise holomorphic $\chi$ which  vanishes at $\pm 1$ and takes the value 1 at origin.   The critical points are $\pm 1/\sqrt 3$ and corresponding critical values $\pm 2/ 3 \sqrt 3$,   and hence $\chi$  
has  a multicentric representation  with $g_j$  given as a converging power series  in $| w | < |2/ 3\sqrt 3| = \rho_{crit} \approx 0,385$.   Let $ r >0$ be small enough so that along $|z|=r$ we have $|p(z)| < \rho_{crit}$, that is  $r (1+r^2) < \rho_{crit}$, then  we have from (\ref{expressderivatives})
\begin{equation}
\alpha_{1,n}  = \ \frac{1}{2\pi i} \int_{|\lambda|=r} \frac{1}{ \lambda^{n+1}  } \  \frac{1}  {(1-z^2)^n} \ d \lambda , 
\end{equation}
while
\begin{equation}
\alpha_{2,n}  = \ \frac{1}{2\pi i} \int_{|\lambda|=r} \frac{1}{ \lambda^{n}  } \  \frac{1}  {(\lambda+1)(1-z^2)^n} \ d \lambda,  
\end{equation}
and

\begin{equation}
\alpha_{3,n}  = \ \frac{1}{2\pi i} \int_{|\lambda|=r} \frac{1}{ \lambda^{n}  } \  \frac{1}  {(\lambda-1)(1-z^2)^n} \ d \lambda.
\end{equation}
The integrals are obtained easily with residue calculus: 
\begin{alignat*}{10}
g_1(w) &= 1 &    &+ 2 w^2 &       &+ 10 w^4 &         &+ 56 w^6 &         &+ 330 w^8  &+ \cdots \\
g_2(w) &=  &  w  &-  w^2  &+ 4w^3 &- 5w^4   &+ 21 w^5 &- 28w^6 &+ 120 w^7 &- 165 w^8 &+ \cdots \\
g_3(w) &=  &- w  &-  w^2  &- 4w^3 &- 5w^4   &-21 w^5  &-28w^6  &-120 w^7 &-165 w^8  &-\cdots
\end{alignat*}

\end{example}
 \bigskip
 
\begin {example} \label{example2}
Let now $p(z)= z (z^2+4z+5)$. 
Here again we  have $\chi(0)=1$ while $\chi$ vanishes at $\lambda_j = -2 \pm i$.  Here the critical radius $\rho_{crit} =2$. 
The components of $g$ are:
\begin{alignat*}{10}
g_1(w) &= &1  &- \phantom{2}\frac{4}{5^2}\phantom{i}w   &+ \phantom{19}\frac{38}{5^4}\phantom{12i}w^2 		  &-
\phantom{40i}\frac{16}{5^4}\phantom{28i}w^3     		  
		&+ \phantom{441}\frac{882}{5^7}\phantom{328i}w^4  &- \dots \\
g_2(w) &= &  &+\frac{2+i}{5^2}w    &- \frac{19+12i}{5^4}w^2  &+\frac{40+28i}{5^5}w^3     
		&- \frac{441+328i}{5^7}w^4  &+ \dots \\
g_3(w) &= &  & +\frac{2-i}{5^2}w   &- \frac{19-12i}{5^4}w^2 & +\frac{40-28i}{5^5}w^3     
		&- \frac{441-328i}{5^7}w^4  &+ \dots \\
\end{alignat*}

\end{example}

In our numerical tests these expansions are truncated. As the expansions converge in discs $|w| \le r < \rho_{crit}$,  we have uniform bounds for the numerical errors, independent of the degree of the approximation, see Section 2.6.

\subsection{Using residue calculus}

For the convenience of the reader we  give the formulas for the residue calculus. 
Suppose $\varphi=1$ inside a component containing $\lambda_j$ and for simplicity,  that $\rho$ is small enough so that all other $\lambda_k$'s are outside this component.  
  We need to calculate the residues at $\lambda_j$ of the terms in the power series of   
\begin{equation}
K_k(z,w) = \frac{1}{z-\lambda_k}\sum_{n=0}^\infty  \frac{w^n}{p(z)^n }.
\end{equation}
In general, if $g(z)$ has pole at $z_0$ of order $m$ so that $m$ is the smallest integer for which $(z-z_0)^m g(z)$ is analytic at $z_0$ then the { \it residue} at $z_0$ is given by the formula
\begin{equation}\label{residueformula}
res(g,  z_0) = \frac{1}{(m-1)!} \lim_{z\rightarrow z_0}  \frac{d^{m-1}}{dz^{m-1}} [ (z-z_0)^m g(z)]
\end{equation}
Now,  in the  series of $K_j(z,w)$ the function multiplying $w^n$  has a pole at $\lambda_j$ of order $n+1$.   Factor $p(z) = (z-\lambda_j) q_j(z)$. Then the residue is obtained as
$$
\alpha_{j,n} = \frac{1}{n!} \lim_{z\rightarrow \lambda_j} \frac{d^{n}}{dz^{n}} q_j^{-n},
$$
 and 
 $$
 \frac{1}{2\pi i }\int_{\gamma_j} K_j(z,w) dz= \sum_{n=0}^\infty \alpha_{j,n} w^n
 $$
 where the series converges for small enough $w$.  In a similar way,  with $K_k(z,w)$ the  term multiplying $w^n$ has a pole at $\lambda_j$ of  order    $n$.  Thus   $\alpha_{k,0}=0$ and for $n\ge 1$
 $$
 \alpha_{k, n} = \frac{1}{(n-1)!} \ \frac{d^{n-1}}{dz^{n-1}} \ \left( \frac{1}{(z-\lambda_k) q_j(z)^n }\right). 
 $$

\bigskip

\subsection{Recursion by differentiating}
\bigskip

If we differentiate the  holomorphic $\varphi$ in (\ref{multicentric})
and assume the data $\varphi^{(\nu)} (\lambda_k)$ to be given, then an explicit recursion for $\alpha_{j, \nu} = f_j^{(\nu)}(0)/ \nu !$ is obtained from 
\begin{equation}\label{explicitrecursion}
p'(\lambda_j)^{\nu} f_j^{(\nu)}(0) = \varphi^{(\nu)}(\lambda_j)- t_{j,\nu},
\end{equation}
where
$$
t_{j,\nu}= \sum_{k=1}^d \sum_{\mu=0}^{\nu-1 } \binom{\nu}{\mu} \delta_k^{\nu-\mu}(\lambda_j) \sum_{l=0}^{\mu} b_{\mu,l}(\lambda_j) f_k^{(l)}(0) +  \sum_{l=0}^{\nu-1} b_{\nu,l}(\lambda_j) f_j^{(l)}(0)
$$
and the polynomials $b_{\nu,l}$  are obtained from a triangular table, see \cite{Ne12a, An22}. 
 
 \subsection{Solving  $g\circledcirc g=g$}

In \cite{Ne16} a  polyproduct $f \circledcirc g$ was defined such that if $f$ and $g$ represent $\varphi$ and  $\psi$,  then $f \circledcirc g$ represents $\varphi \psi$. 
Since $\chi^2= \chi$ we can obtain the expansions for the components of $g$  by subsituting  an Ansatz for each of the component,  which leads  again to an explicit recursion for the coefficients.  For $g_k(w)= \sum_{n=1}^\infty \alpha_{k,n}w^n$ we get

\begin{equation} \label{yleinen_rekursio}
\alpha_{k,n+1}(1-2\alpha_{k,0}) = \sum_{l=1}^{n} \alpha_{k,l}\,\alpha_{k,n+1-l} - 
\sum_{m\neq k} \sigma_{km} \sum_{l=0}^{n}
(\alpha_{k,l}-\alpha_{m,l})(\alpha_{k,n-l}-\alpha_{m,n-l}).
\end{equation}

In the symmetrical case of $\chi$ the starting values are $\alpha_{1,0}=1$ and $\alpha_{2,0}=\alpha_{3,0}=0$.
For the first step ($n=0$) \eqref{yleinen_rekursio} simplifies to
\begin{eqnarray*}
\alpha_{1,1}\underbrace{(1-2\alpha_{1,0})}_{=-1} &=& -\sum_{m\neq 1} \sigma_{1m} (\alpha_{1,l}-\alpha_{m,l})^2 \\
 &=& \sigma_{1,2} + \sigma_{1,3} 
 = -\frac{1}{2} + \frac{1}{2} 
 = 0,
\end{eqnarray*}
and for $k=2$ 
\[
\alpha_{2,1}\underbrace{(1-2\alpha_{2,0})}_{=1} 
 = \sigma_{2,1}\cdot 1^2 + \sigma_{2,3}\cdot 0^2 
 = -1,
\]
i.e. $\alpha_{2,1} = -1$ and similarly for $k=3$ we get $\alpha_{3,1} = 1$.

 \subsection{ Formulas for $z^k$}

 In case one  has  the plain expression $H_n(z)$  available  with exact  coefficients,  it is desirable to transform it to  truncated multicentric form $M_n$.  To that end we list the formulas for transforming the powers  $z^k$. 
 Denote by $Z^{\circledcirc k}$ the vector function representing $z^k$, so that
 
 $$
 z^k = \sum_{j=1}^d \delta_j(z) \ (Z^{\circledcirc k})_j (p(z)).
 $$

Using the polyproduct we have from $Z^{\circledcirc (k+1)} =  Z  \circledcirc Z^{\circledcirc k}  $

 \begin{equation} \label{zkrecursion}
 (Z^{{\circledcirc (k+1)}})_j(w) = \lambda_j (Z^{{\circledcirc k}})_j(w) + w \sum_{l=1}^d \frac{(Z^{\circledcirc k})_l(w)}{p'(\lambda_l)}, 
 \end{equation}
 see \cite{Ve25}.
 
 For example in the symmetrical example $p(z)=z(1-z^2)$, the Hermite polynomial with $n=2$ is 
 $H_2(z) = (1-z^2)^2 = 1 -2z^2+z^4$. The multicentric
 representation for $z^2$ comes directly from Lagrange interpolation formula, 
 \[
 z^2 = \sum_{j=1}^3 \delta_j(z) \lambda_j^2,
 \] 
 and the powers $z^3$ and $z^4$ are calculated with the recursion formula  as follows,
 \[
 z^3 =  \sum_{j=1}^3 \delta_j(z) \left(
 \lambda_j^3 + w \sum_{l=1}^3 \frac{\lambda_l^2}{p'(\lambda_j)}  
 \right)
  =  \sum_{j=1}^3 \delta_j(z) \left(\lambda_j^3 - w  \right)
 \]
 and
 \[
 z^4 =  \sum_{j=1}^3 \delta_j(z) \left(
 \lambda_j(\lambda_j^3 - w) + w \sum_{l=1}^3 \frac{\lambda_j^3 - w}{p'(\lambda_j)} 
 \right)
  = \sum_{j=1}^3 \delta_j(z) \left( \lambda_j^4 - \lambda_j w \right).
 \]
 Combining the results we have $M_2$ representing $H_2$ in the form
 \[
 M_2(z) =  \sum_{j=1}^3 \delta_j(z) \left( 1 - 2\lambda_j^2 + \lambda_j^4 - \lambda_j p(z)
 \right).
 \]
 
\subsection{Errors due to inexact evaluation of the variables} \label{subsectionerror}

We can use the integral representation from Section~\ref{cauchysection} to derive bounds for the accumulated errors.
Let $\varphi$ be holomorphic for $|p(z)| < r$ and  denote for $\rho<r$
$$
M(\rho) = \max_{|p(z)| \le \rho} |\varphi (z)|. 
$$
We assume  here that the exact values of coefficients $\alpha_{k,j}$  can be used, as this is the case e.g. when  the roots  $\lambda_k$ have rational real and imaginary parts  and $\varphi$ takes  only integer values. 

Let us denote $$D(\rho)= \max_k \max_{|p(z)| \le \rho} |\delta_k(z)|, $$ and  $L(\rho)=  \sum_{k=1}^d L_k(\rho)$ where
$$
L_k(\rho)= \frac{1}{2\pi} \int_{\gamma_\rho} \frac{|d\lambda|}{|\lambda - \lambda_k|}.
$$
 Then we have from (\ref{expressderivatives})
 $$|\alpha_{k,j}| \le \  M(\rho)\ L_k(\rho) \ \rho^{-j}. 
 $$
Consider now bounding $f_k$ for $|w|\le  \rho_0$ where $\rho_0< \rho.$
We have immediately
$$
\max_{|w|\le \rho_0} | f_k(w)| \le M(\rho) \ L_k(\rho) \  \frac{\rho}{\rho - \rho_0}.
$$
As the coefficients $\alpha_{k,j}$ can be computed exactly, we  focus on the effect caused by nonexact evaluation of the variables $z$  and $w=p(z)$.  We denote by their inexact evaluations by $\hat z$   and by $\hat w$ and  allow the approximations to differ each time when applied.  

Let $\rho_0 <\rho_1 < \rho$.  We require the approximations to satisfy the following. On the  evaluation of basis polynomials  we require 

\begin{equation}\label{delttavirhe1}
\max_{|p(z)|\le \rho_0} |\delta_k(\hat z)|  \le \max_{|p(z)| \le \rho_1} |\delta_k(z)|\end{equation} 

\begin{equation}\label{delttavirhe2} \max_{|p(z)| \le \rho_0}|\delta_k(z) - \delta_k(\hat z)| \le {\varepsilon_0}
\end{equation}
  and  on the evaluations of the polynomial variable 
\begin{equation} \label{peevirhe1}
\max_{|p(z)| \le \rho_0} |\hat w|  \le \rho_1,
\end{equation}
  
\begin{equation}\label{peevirhe2}
\max_{|p(z)|\le \rho_0} |p(z) - \hat w| \le \varepsilon_1.
\end{equation}
Finally, for the approximation of $\varphi$  at the point $z$  we  write 
$$ \hat \varphi (z) = \sum_{k=1}^d \delta_k(\hat z) f_k(\hat w).  
$$

Then we  have the following bound for the total error in $\varphi(z)$ inside the lemniscate $|p(z)| = \rho_0$. 

\begin{proposition}  If $|p(z)| \le \rho_0$,  and the assumptions above hold, we have 
\begin{equation}\label{kokovirhe}  
|\varphi(z) - \hat \varphi (z) | \le M(\rho) L(\rho) \  \{ \frac{\rho}{\rho - \rho_0}
\ \varepsilon_0 + 
D(\rho_1)  \frac{\rho_1}{(\rho - \rho_1)^2} \ \varepsilon_1 \ \}.
\end{equation}
\end{proposition}

\begin{proof} 
 We have 
\begin{multline}\label{kakstermia}
|\varphi(z) - \hat \varphi(z)|  
  \le   \sum_{k=1}^d |\delta_k(z) - \delta_k(\hat z)| \ 
|f_k(p(z))| \\ + \sum_{k=1}^d |\delta_k(\hat z)| \  | f_k(p(z)) - f_k(\hat w)|.
\end{multline}
The first term on the right of (\ref{kakstermia}) gives immediately the first term on the right of (\ref{kokovirhe}).  To get the second term  write
$$
|p(z)^j - \hat w^j| \le (j-1)\  \rho_1^{j-1} \varepsilon_1.
$$
 
\end{proof}

\begin{remark}
If the series expansions are truncated, then the error  bound holds for the truncated versions  $M_n$ without modifications, or if one wants, the  infinite series can be replaced by their, smaller, truncated versions.  
\end{remark}

 \section{Formulas for related Hermite interpolation}

Denoting by $\delta_{k,j,n}$ the  basis polynomials  of degree $N=d(n+1)-1$ satisfying
for $k,l=1,\dots, d$ and $j, m=0,\dots,n$
\begin{equation}
 \delta_{k,j,n}^{(m)}(\lambda_l) = 1 \text{ when } \ k=l \text{ and }  j=m, \text{while vanish otherwise}    
\end{equation}
we can write the  general Hermite polynomials as linear combinations of these, see 
(\ref{hermiteyleinen}).

 We restrict our numerical studies in the particular case of just one basic polynomial, that of $\delta_{1,0,n}$.  This  basis polynomial has a simple form and we compare the  numerical stability of evaluating it as $n$ grows, with the representation obtained from truncating the multicentric representation. 

 We begin with a simple observation of the form of $\delta_{1,0,n}$, which we denote by $S_n$.  Here, as before $\delta_1$ denotes the Lagrange interpolation polynomial taking value 1 at $\lambda_1$.

 \begin{proposition} We have 
 \begin{equation}\label{special}
 S_n= \delta_1^{n+1} \lfloor \delta_1^{-n-1}\rfloor_n.
 \end{equation}
 \end{proposition}

\begin{proof}
We may assume that $\lambda_1=0$. 
Clearly $  \delta_1(z)^{n+1} \  \lfloor  \delta_1^{-n-1}  \rfloor _n(z)$ vanishes at least $n+1$ times at every $\lambda_k$  with $k\ge 2$.    On the other hand,  for small $z$ we have
$$
\delta_1(z)^{n+1} \lfloor  \delta_1^{-n-1}  \rfloor _n (z) = 1 + \mathcal O (z^{n+1})
$$
and hence  it is the basis polynomial $\delta_{1,0,n}$. 
\end{proof}

\begin{remark}
Since $\delta_1(z) = 1 - \sum_{k=2}^d \delta_k(z)$ we have
$$
\delta_1(z)^{-n-1} = \sum_{m=0}^{\infty} \binom{n+m}{n} Q(z)^m
$$
where $Q(z)= \sum_{k=2}^d \delta_k(z)$.
\end{remark}

\begin{example}  
Returning to the special case of $p(z) = z (1-z^2)$ we obtain for example
$$
S_3(z) = (1-z^2)^4 (1+ 4z^2) \ \text{and } \ S_4(z) = (1-z^2)^5 (1+5z^2 + 15 z^4).
$$
\end{example}

In \cite{Ke12} an algorithm  to evaluate Hermite polynomials, particularly suitable for parallel computation is given.  Modifying it into our test case it reduces to the form
 \begin{equation}\label{parallel}
 T_n = \delta_1^{n+1} \sum_{k=1}^d s_{n+1,k}
 \end{equation}
where the polynomials $s_{n+1,k}$ are given  as follows:
$$
   s_{n+1,k}(z) = \left[\begin{matrix} 1 & \frac{z}{1!} & \dots & \frac{z^{n}}{n!} \end{matrix} \right]
\left[\begin{matrix} c_{j,1}\\ \vdots\\ c_{j,n+1}\end{matrix} \right]
 = c_{j,1} + \frac{c_{j,2}}{1!}\,z + \dots + \frac{c_{j,n+1}}{n!}\,z^n.
$$
The vector $[c_{j,m}]$ is obtained as follows. 
 
Kechriniotis et al. derive in \cite{Ke12} (in our notation) that, 
\[
h_n(z) = \sum_{j=1}^d \sum_{k=1}^{n+1} X_{j,n+1} (I_{n+1} - \Lambda_j)^{k-1} A_j,
\]
for the general case $h_n^{(j)}(\lambda_i)=a_{ij}$. Here 
\[
X_{j,n+1} = \delta_j(z)^{n+1} \left[ 1 \quad \frac{(z-\lambda_j)}{1!} \quad \dots \quad \frac{(z-\lambda_j)^{n}}{n!}\right],
\]
and $\Lambda_k$ is a lower triangular matrix $[l_{k,ij}]$ with
\[
l_{k,ij} = \binom{i-1}{j-1} L_{k,n+1}^{(i-j)}(\lambda_k)
\]
and $L_{k,n+1}(z) = \delta_k(z)^{n+1}$,
and $A_j$ is the $j^{th}$ column of data from $[a_{ij}]$.

For $\chi$ this simplifies to
\[
T_n(z) = \delta_1^{n+1}(z) \cdot \sum_{k=1}^{d} \underbrace{
\left[\begin{matrix} 1 & \frac{z}{1!} & \dots & \frac{z^{n}}{n!} \end{matrix} \right]
\left[\begin{matrix} c_{k,1}\\ \vdots\\ c_{k,n+1}\end{matrix} \right]
}_{ = s_{n+1,k}(z)},
\]
where  $[c_{j,m}]$ is the first column of matrix $(I_{n+1}-\Lambda_1)^j$.

 \bigskip

   
\section{Numerical experiments}
 
We denote by 
\begin{equation} \label{Mn}
 M_n(z) = \sum_{k=1}^d \delta_k(z)  \lfloor g_k\rfloor _n (p(z)),
\end{equation}   
\begin{equation} \label{Sn}
 S_n(z) = \delta_1(z)^{n+1}  \lfloor \delta_1^{-n-1} \rfloor _n  (z)
\end{equation} 
and  
\begin{equation} \label{Tn}
 T_n(z) = \delta_1(z)^{n+1} \sum_{k=1}^d s_{n+1,k}(z)
\end{equation} 
the {\it same} polynomial 
\begin{equation} \label{Hn}
 H_n(z) = \sum_{j=0}^N a_j z^j,
\end{equation}
written in different ways, in order to compare the accumulation of different errors. 
Here all polynomials are of  degree $N\le d(n+1)-1$  but we list them to be indexed by $n$, 
i.e. the number of rounds of 
derivative data used. Thus $p$ (and hence $d$) are kept fixed.   

Each of the component polynomials, e.g. $p(z)$, $\lfloor \delta_1^{-n-1} \rfloor_n (z)$, etc.
and in the case of Hermite polynomial $H_n(z)$, are evaluated using Horner's method.  
We will add two kinds of errors to the Horner's method.

First we add an \textit{evaluation error} of size $\nu$ to
the point of evaluation $z$. It is added once to the test point $z$. 
For example instead of $az^2+bz+c = c + z(b +z\cdot a)$ we use
\[
az^2+bz+c \approx c + \hat{z}(b + \hat{z}a).
\]
Here $\hat{z}$ is of form $z+(a+ib)$ where $|a|<\nu$ and $|b|<\nu$.

The second source of error added to the Horner's method is the \textit{rounding error} of size $\mu$ 
after each arithmetic operation. We use the formula
\[
b + \hat{z}a \approx \left(\,b + (\hat{z}a)(1+\mu_1)\,\right)(1+\mu_2),
\]
for the error, where $\mu_{i}$:s, $i=1,2$, 
are random complex numbers with real and imaginary parts of size less than $\mu$.

The third source of errors comes from the location of the evaluation point $z$. 
As we saw in section \ref{subsectionerror} the error in inexact evaluation 
can grow larger with the distance from the roots, i.e. the levels $\rho$ of the lemniscate $\gamma_\rho$. 

The fourth source of errors present with all previous sources 
is the double precision floating point arithmetics used in the calculations.
To minimize other sources of errors we calculate the coefficients of the component polynomials 
and the exact value $H_n(z)$ algebraically
before converting them to floating point numbers.

\bigskip

\subsection{Testpoints} \label{testcases}

We use two previously mention cases in Example \ref{example1} and Example \ref{example2} for
our numerical examples. For simplicity we use the names Example 1 and Example 2 from now on.
Images of the examples are in Figures \ref{lemniskaatta_ex1} and 
\ref{lemniskaatta_ex2}. In both cases the $\rho_{crit}$ refers to the critical level, i.e. to the
lemniscate 
$\gamma_\rho$, where the area with $\chi \equiv 1$ makes a contact with the area(s) where $\chi\equiv0$.

\begin{figure}
\begin{center}
\includegraphics[scale=0.25]{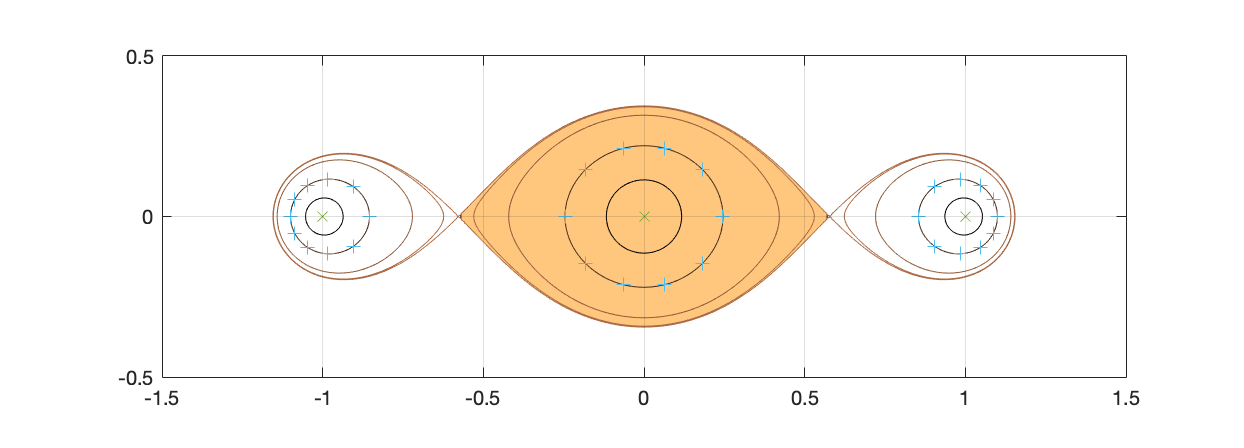}
\end{center}
\caption{(Example 1) Level curves of polynomial $p(z)=z(1-z^2)$ with roots $\lambda_1=0$, $\lambda_2=-1$, and $\lambda_3=+1$. The roots are marked with 'x'. 
The critical points $c$, from $p'(c)=0$, are $c=\pm 1/\sqrt{3} \approx \pm 0.577$. 
At corresponding level $|p(c)| = 2/(3\sqrt{3}) =:\rho_{crit} \approx 0.3849$ the three components join together. 
The piecewise constant function $\chi$ we approximate with Hermite polynomials 
is 1 at $\lambda_1=0$ and 0 at the other two roots.
The component enclosing $\lambda_1=0$ is colored.
Four sets of smaller level curves $\gamma_\rho$ where $\rho$ is $t \times \rho_{crit}$ and $t$ is either $0.30$,
$0.60$, $0.90$ or $0.99$ are drawn in addition to the lemniscate $|p(z)|=\rho_{crit}$. 
Example of the testpoints used in our numerical calculations are marked with '+' on the lemniscate 
$|p(z)|=0.6\times\rho_{crit}$.}
\label{lemniskaatta_ex1}
\end{figure}

\begin{figure}
\begin{center}
\includegraphics[scale=0.25]{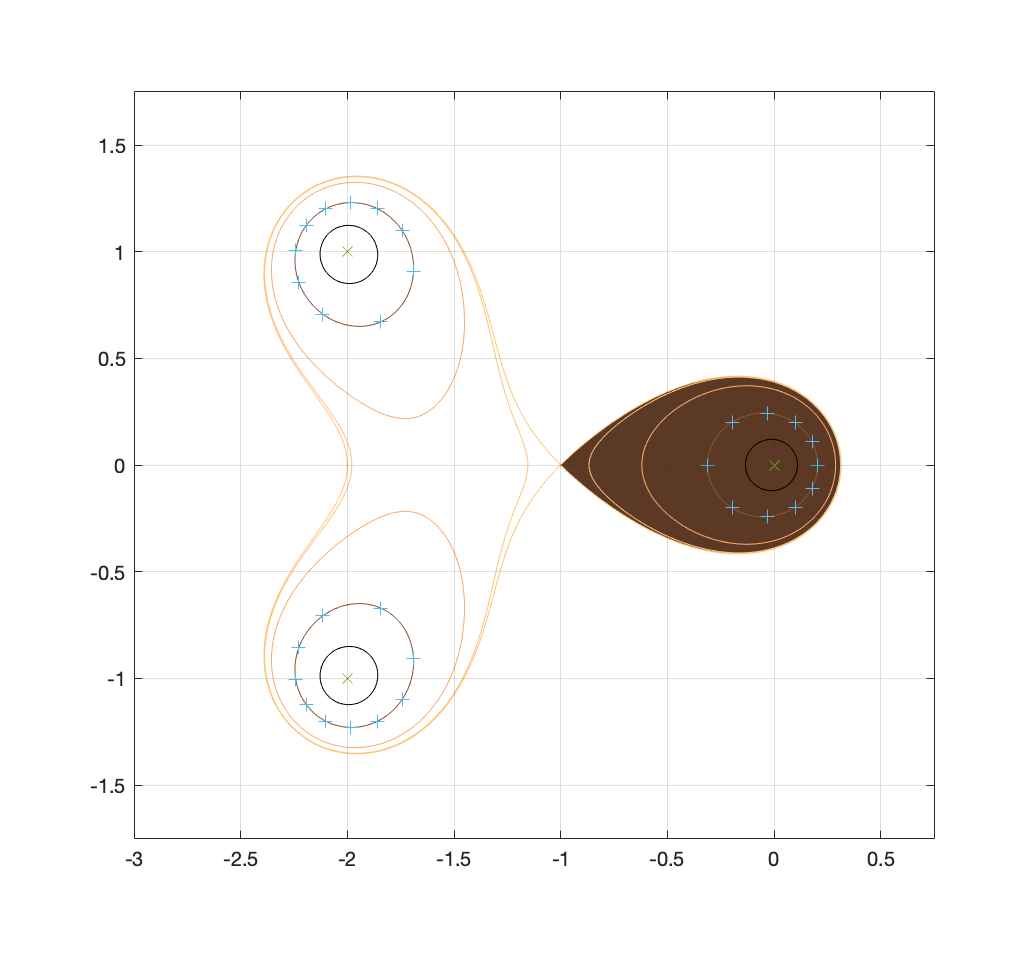}
\end{center}
\caption{(Example 2) Level curves of polynomial $p(z)=z(z^2+4z+5$) with roots $\lambda_1=0$, $\lambda_2=-2+i$, and $\lambda_3=-2-i$.
The roots are marked with 'x'. 
There are two critical points $c$, $p'(c)=0$, $c_1=-1$ and $c_2=-5/3$. 
At corresponding level $|p(c_1)| = 2 =:\rho_{crit}$ the two remaining components join together. 
The component including the root $\lambda_1=0$ is colored. 
Four sets of smaller level curves $\gamma_\rho$ where $\rho$ is $t \times \rho_{crit}$ and $t$ is either $0.30$,
$0.60$, $0.90$ or $0.99$ are drawn in addition to the lemniscate $|p(z)|=\rho_{crit}$.
Example of the testpoints used in our numerical calculations are marked with '+' on the lemniscate 
$|p(z)|=0.6\times\rho_{crit}$.}
\label{lemniskaatta_ex2}
\end{figure}

For each $\rho$ and lemniscate $\gamma_\rho$ we pick only a subset of test points, where we compare
the exact value $H_n(z)$ to the approximate values with added errors. First we take $k=10$ 
equally spaced points $\phi$ from the interval $[0,1]$. To get points from the lemniscate we solve
the roots of
\[
p(z) = \rho \cdot e^{i2\pi\cdot \phi},
\]
with Matlab's 'root' command. This returns $d=3$ roots for each $\phi$, i.e. we have 30 test points to examine
for each $n$ and $\rho$. To compare the effect of the size $\rho$ to the maximum error at these 30 points
 we use levels in the Table \ref{rhootaulukko}. The corresponding lemniscates are drawn 
 in Figures  \ref{lemniskaatta_ex1} and  \ref{lemniskaatta_ex2}.

\begin{table}
\begin{center}
\begin{tabular}{|r|r|r|}
\hline 
$\rho$   & Example 1 & Example 2 \\
\hline \hline
$0.3 \times \rho_{crit}$  & $0.1155$ & $0.60$ \\
\hline 
$0.6 \times \rho_{crit}$  & $0.2309$ & $1.20$  \\
\hline 
$0.9 \times \rho_{crit}$  & $0.3464$& $1.80$ \\
\hline 
$0.99 \times \rho_{crit}$  & $0.3811$ & $1.98$  \\
\hline 
\end{tabular}
\end{center}
\caption{In Example 1 $\rho_{crit}$ is $2/3 \sqrt 3 \approx 0.3849$ and in Example 2 $\rho_{crit}$ is $2$.
The images of lemniscates $\gamma_\rho$ are drawn in Figure \ref{lemniskaatta_ex1} and 
\ref{lemniskaatta_ex2} respectively.}
\label{rhootaulukko}
\end{table}

At each test point we calculate the exact value $H_n(z)$ (turned to a double precision floating point number) 
and the numerically computed (with added errors) of each version of the 
\eqref{Mn}--\eqref{Hn}. 

\subsection{Results}

We test each of the variables $\rho$, $\nu$, and $\mu$ separately for both Examples 1 and 2. In each case
we compare the maximum error of each algorithm at the test points for each $\rho$. 
In all figures the amount of the derivative data $n$ 
on the $x$-axis will have values $4:4:24$. The values between these would only add a little variation to graphs, 
but they would not change the behavior between algorithms. 
In the graphs $y$-axis is logarithmic and errors larger then 1 are not drawn.

\subsubsection{Level curves $\gamma_\rho$} \label{rhoo}

In Figure \ref{rhoon_vaikutus} the values $\nu$ and $\mu$ are both $0$. We see the effect of rounding errors
from double precision Horner's algorithm and the error from the location of the test points from
different lemniscates $\gamma_\rho$. 
We use four different levels $\rho$ in the numerical examples. The test points are taken from level curves
 where $\rho$ is $t \times \rho_{crit}$ and $t$ is either $0.30$,
$0.60$, $0.90$ or $0.99$. The numerical values for both examples are listed in Table~\ref{rhootaulukko}.

In all cases the error in Hermite polynomial $H_n(z)$ grows exponentially as expected. 
In the Parallel case $T_n(z)$ errors stay small only near the roots.
Otherwise the growth in errors exceeds exponential growth and grow even faster as $\rho$ grows. 
The Special version grows exponentially as well, but in a much slower fashion compared to Hermite and Parallel versions.
The Multicentric version stays practically constant in all cases.

Even with small value of $\rho$ we can see that the double precision arithmetic is the larger 
source of error in evaluations. 
The errors caused by $\nu$ and $\mu$ will not be visible with a large value of $\rho$.
We will use $\rho = 0.02 \times \rho_{crit}$ in the following tests for the effect of $\nu$ and $\mu$.

\subsubsection{Evaluation error $\nu$} \label{aar}

We use $\mu=0$ for rounding error, $\rho = 0.02 \times \rho_{crit}$ and $n=4:4:24$. The results are in Figure \ref{nyyn_vaikutus}.

In all cases Multicentric stays within reasonable error size 
i.e. from $1\times\nu$ to $10\times\nu$ on the logarithmic scale. Parallel and Special show similar behavior. 
The error stay almost constant as $n$ grows. 
In Example 1 in the two largest values of $\mu$ the Special and Parallel algorithms outperform the Multicentric algorithm.

\subsubsection{Rounding error $\mu$} \label{myy}

We use $\nu=0$ for evaluation error, $\rho = 0.02 \times \rho_{crit}$ and $n=4:4:24$.
The results are in Figure \ref{myyn_vaikutus}.

In all cases Multicentric stays again within reasonable error size i.e. $(1-10)\times\nu$ on the logarithmic scale
unlike the other algorithms. The error growth is modest
for Special and Parallel but do exceed the limits $10\times\nu$.

%
\subsection{Conclusions}

To calculate $\varphi(A)$ using traditional methods we need to know the eigenvalues exactly. 
The benefit of the Multicentric representation for $\varphi$ is that we can use 
approximate values for the eigenvalues.  
To compensate the lack of knowledge we do need to calculate the expansion for larger $n$, but as
we can see from Figure~\ref{n_100} for $n=4:8:100$ the Multicentric version
 keeps the error at constant and approximately the same size as
the added errors. This allows to use as much derivative data as needed.
For Multicentric algorithm it is possible to compute the polynomials used numerically without additional errors, but not for the other algorithms, see Figure~\ref{coeffs_numerically}.

\bibliographystyle{plain}
\bibliography{piecewise-2025-11-10}

\begin{figure}[h]
\newcommand{\imagescale}{0.17}
\begin{center}
\includegraphics[scale=\imagescale]{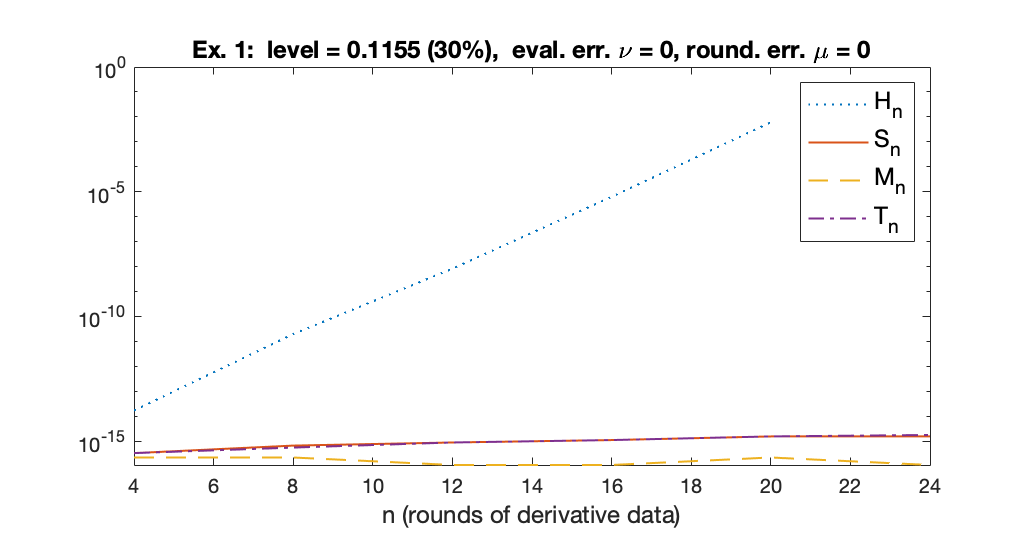}\includegraphics[scale=\imagescale]{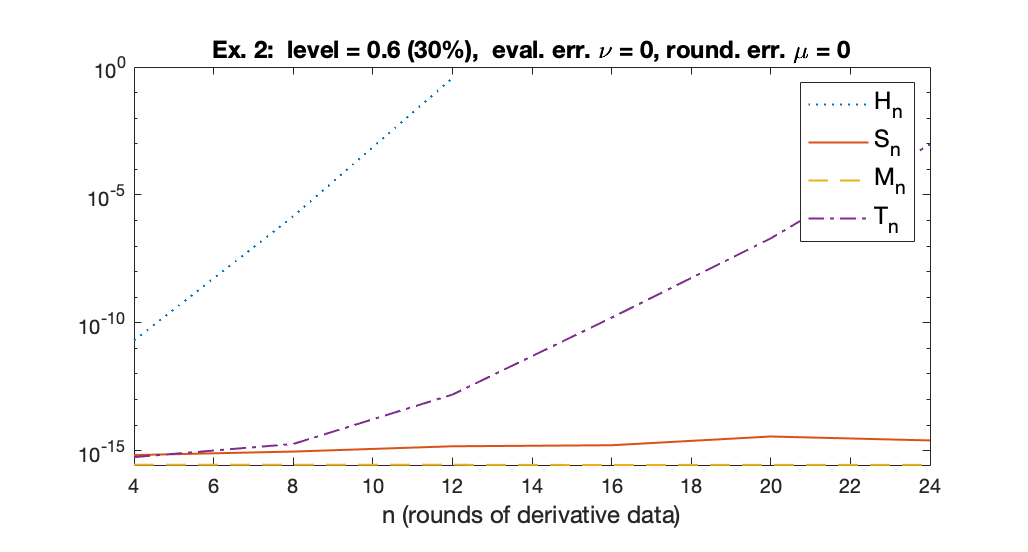}
\includegraphics[scale=\imagescale]{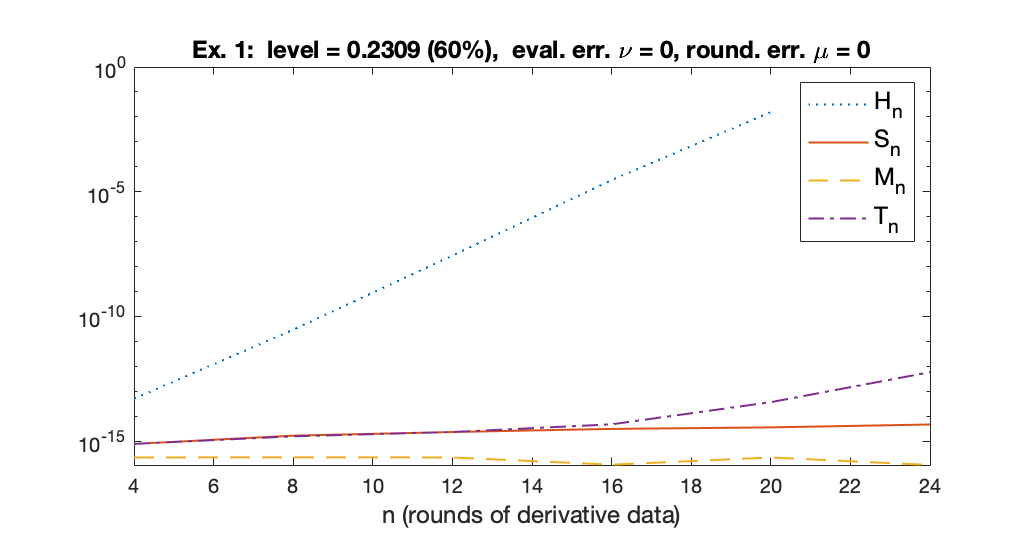}\includegraphics[scale=\imagescale]{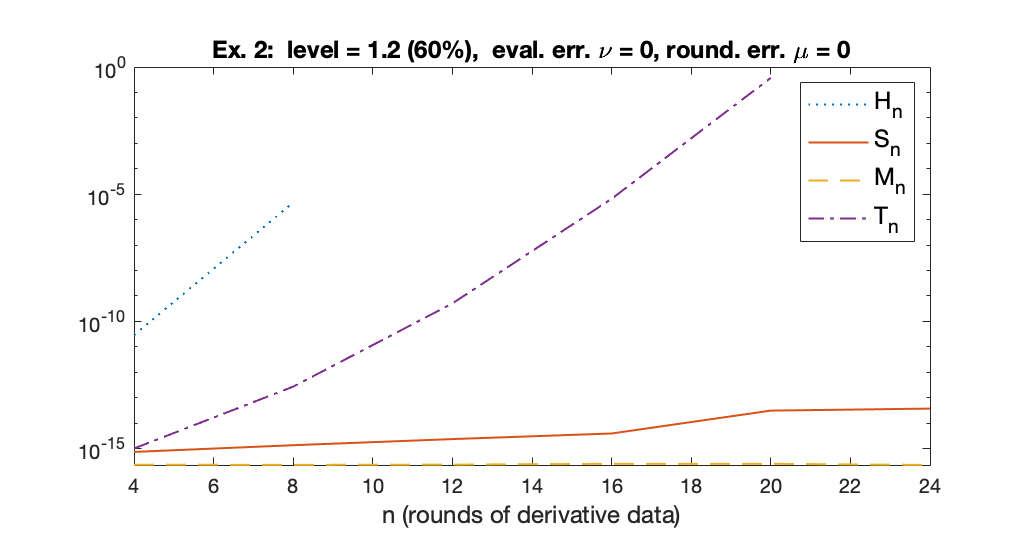}
\includegraphics[scale=\imagescale]{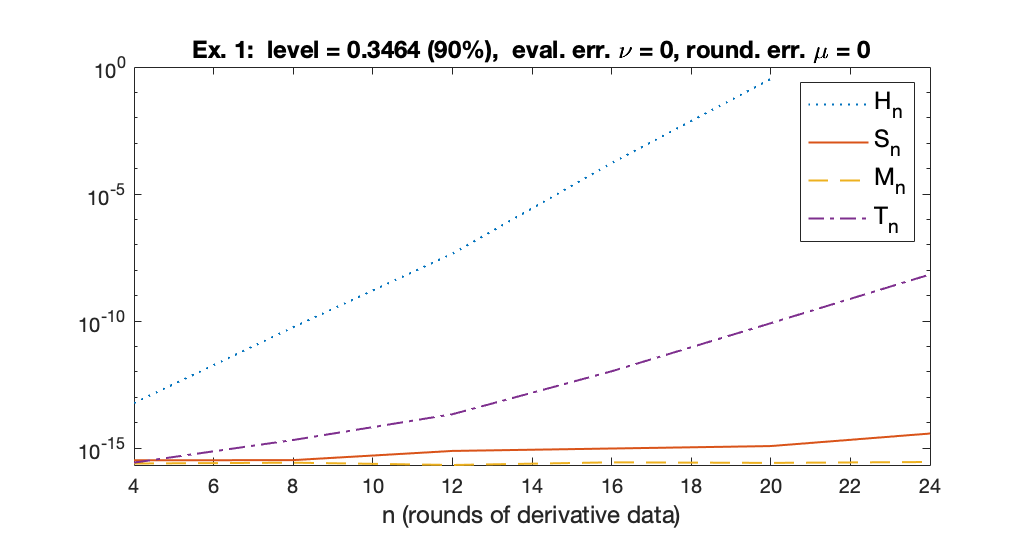}\includegraphics[scale=\imagescale]{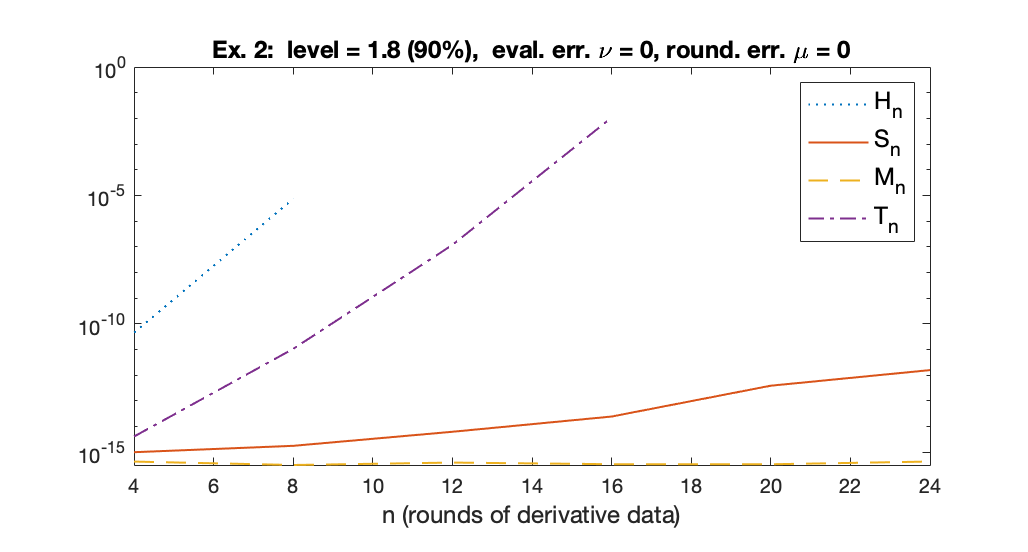}
\includegraphics[scale=\imagescale]{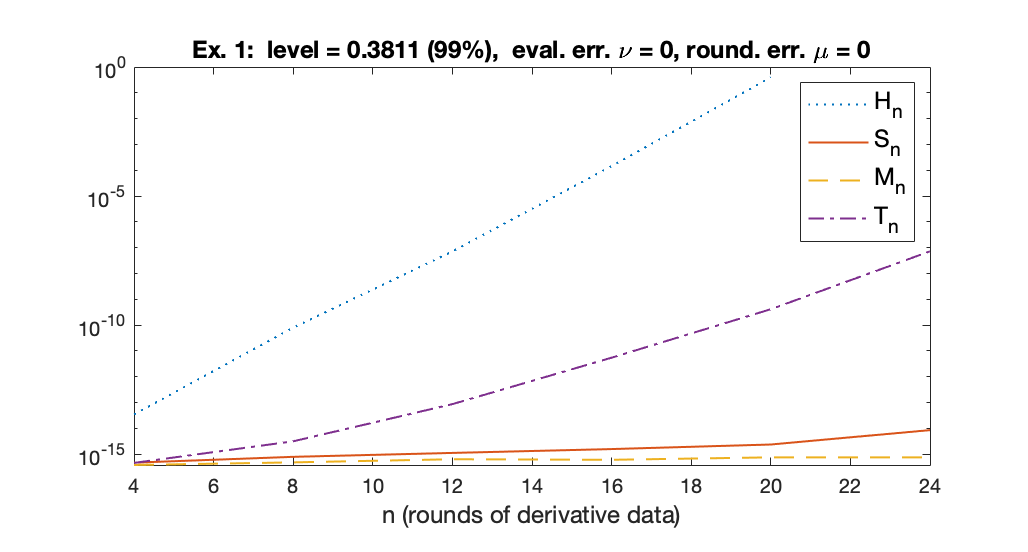}\includegraphics[scale=\imagescale]{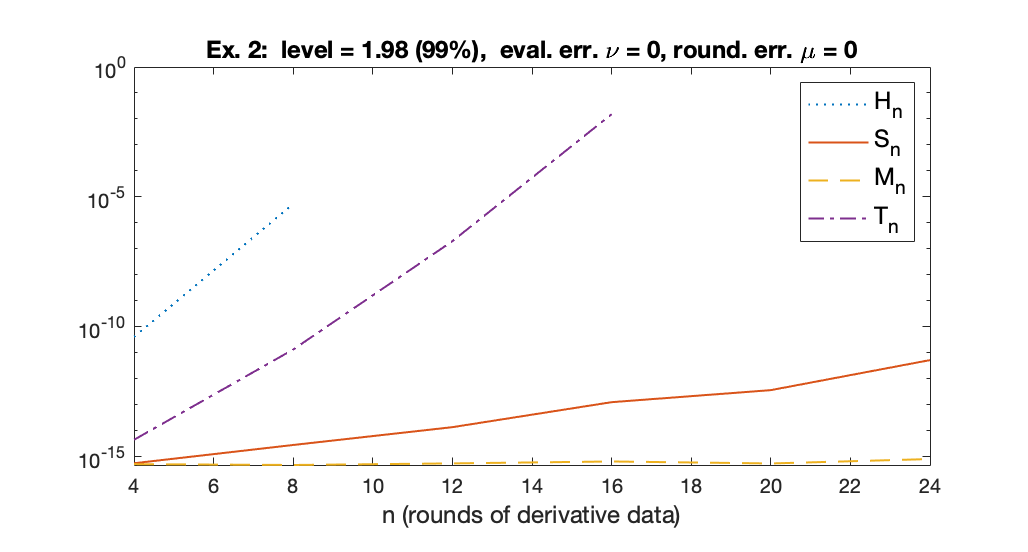}
\end{center}
\caption{(Example 1 in the left column and Example2 on the right hand side) Effect of the size of $\rho$. The errors $\nu$ and $\mu$ are zero. 
The smallest level $\rho=0.3\times\rho_{crit}$ is in the first row and the levels grow with each line according the
Table \ref{rhootaulukko}.} 
\label{rhoon_vaikutus}
\end{figure}

\begin{figure}[h]
\newcommand{\imagescale}{0.17}
\begin{center}
\includegraphics[scale=\imagescale]{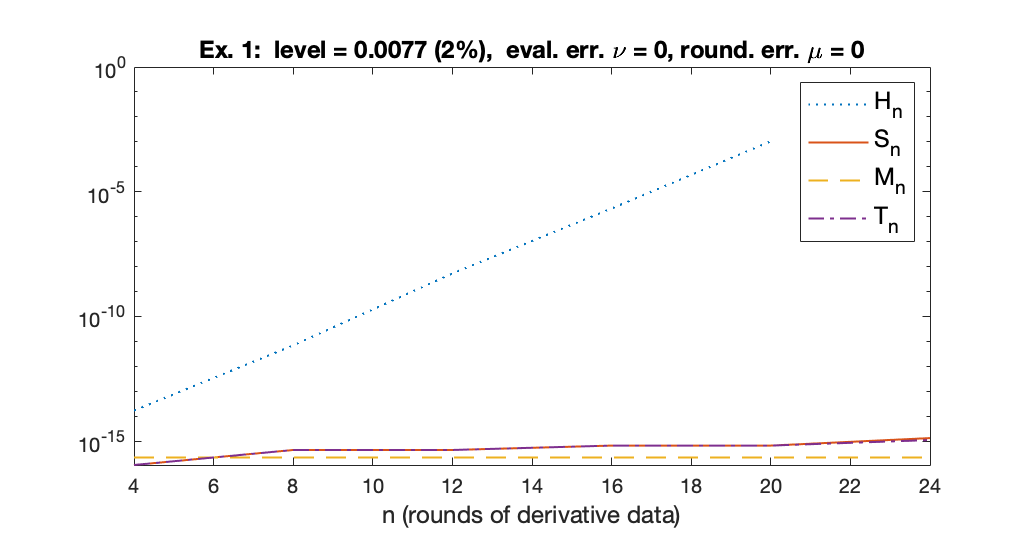}\includegraphics[scale=\imagescale]{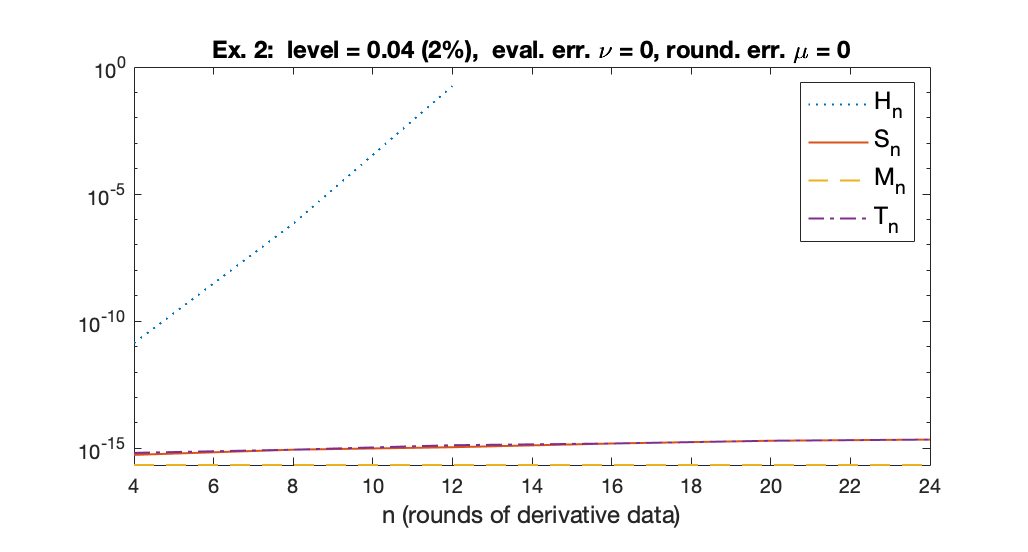}
\includegraphics[scale=\imagescale]{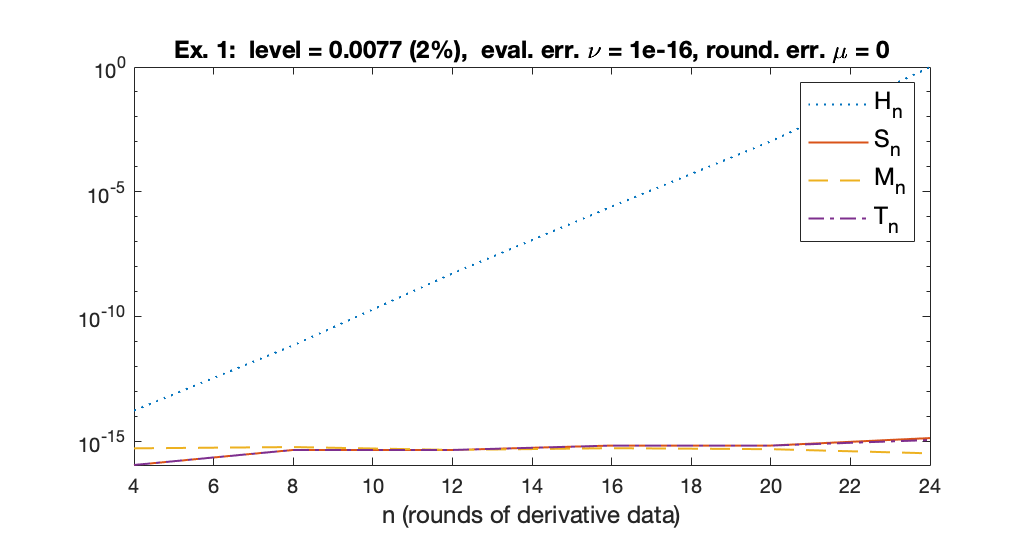}\includegraphics[scale=\imagescale]{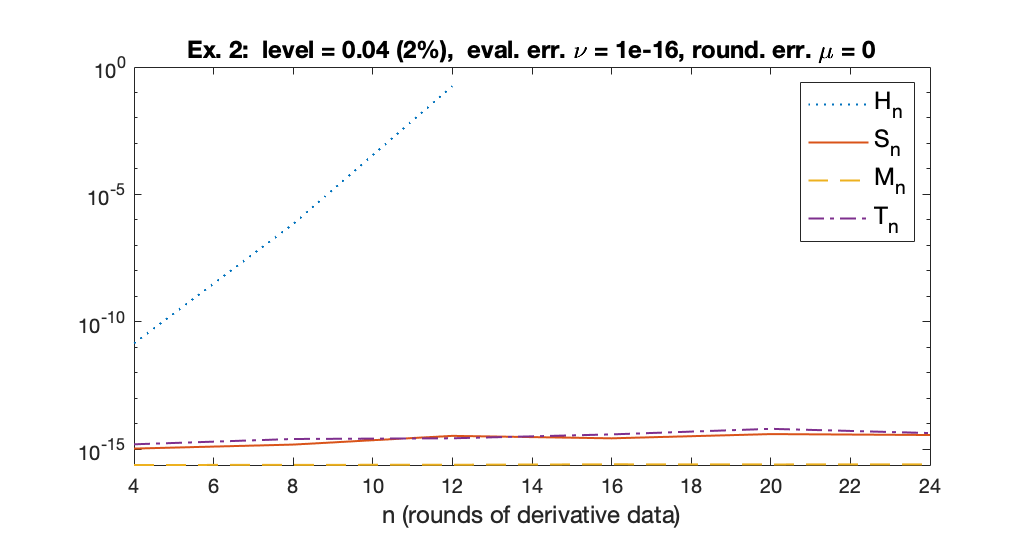}
\includegraphics[scale=\imagescale]{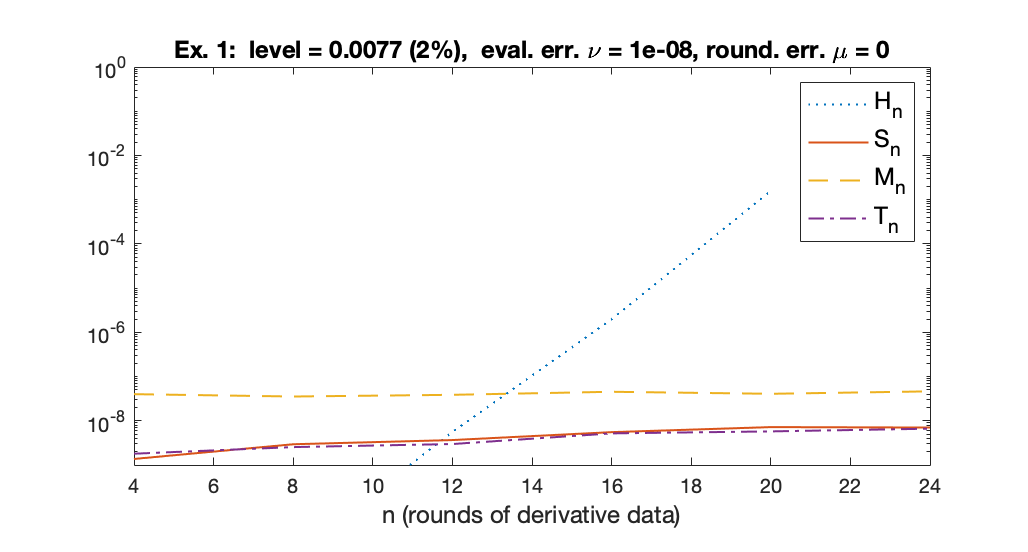}\includegraphics[scale=\imagescale]{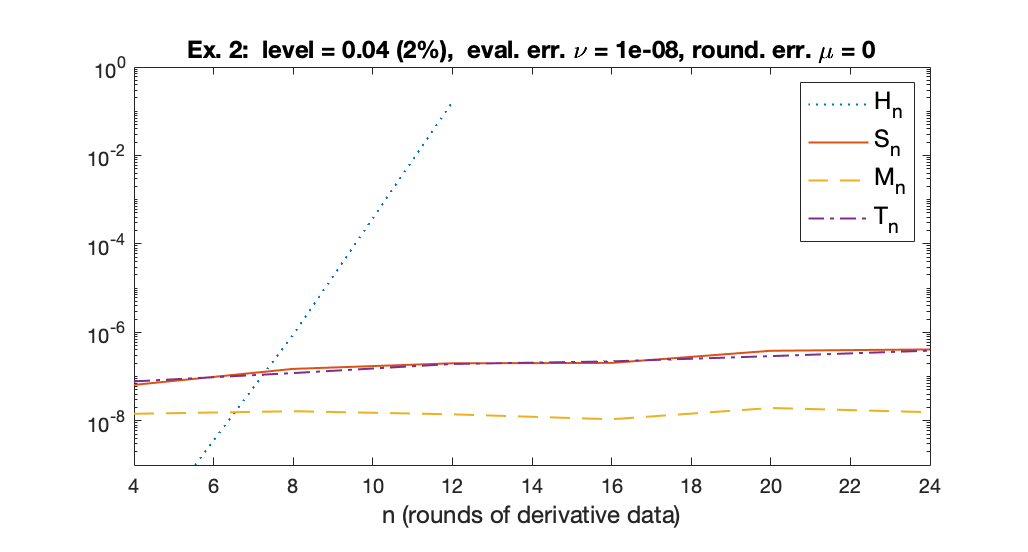}
\includegraphics[scale=\imagescale]{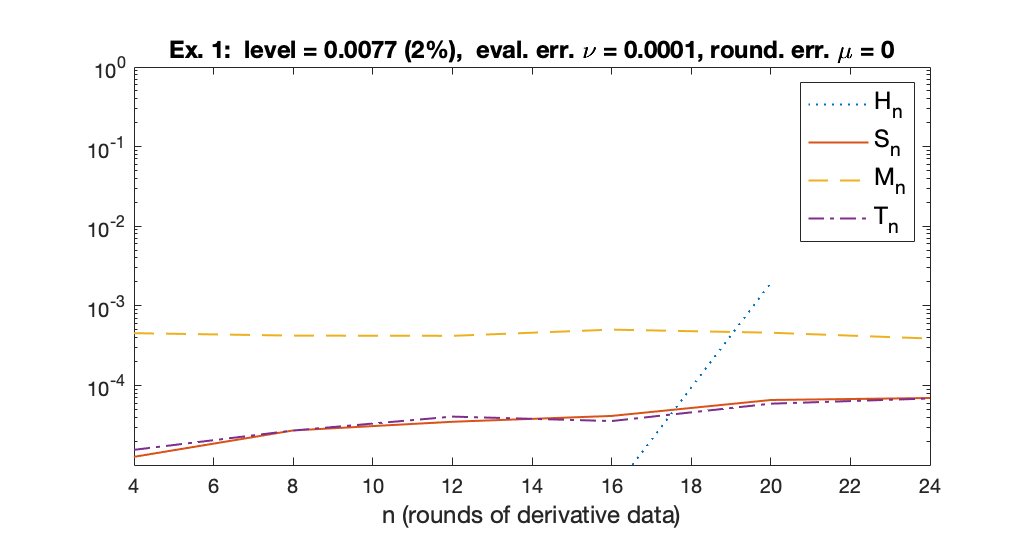}\includegraphics[scale=\imagescale]{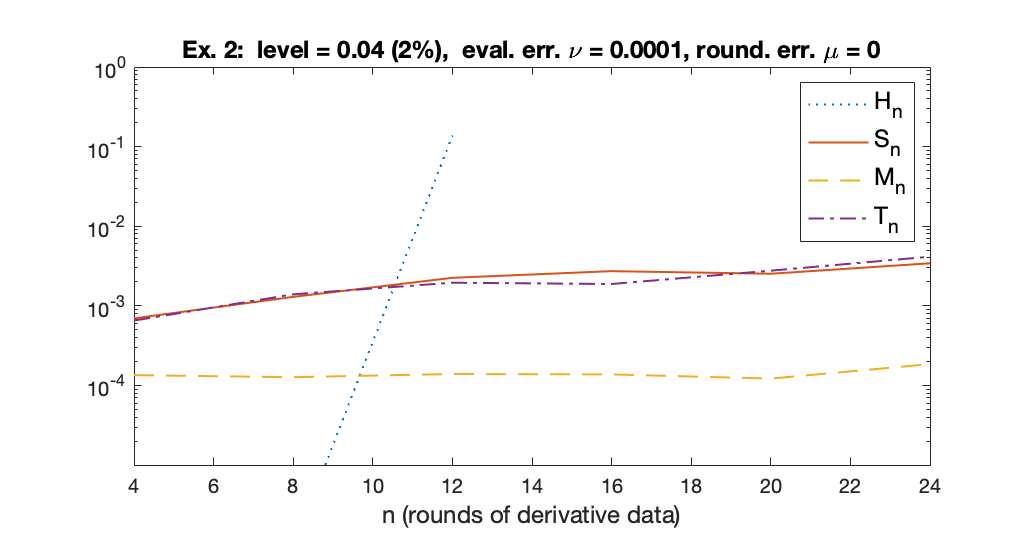}
\end{center}
\caption{(Example 1 on the left and and 2 in the right hand side column) Effect of evaluation error $\nu$ when $\rho=0.02\times \rho_{crit}$ and rounding error $\mu=0$.
The tested values for $\nu$ are $0$, $10^{-16}$, $10^{-8}$, and $10^{-4}$ (from top to bottom).}
\label{nyyn_vaikutus}
\end{figure}

\begin{figure}[h]
\newcommand{\imagescale}{0.17}
\begin{center}
\includegraphics[scale=\imagescale]{a_testiu_rho_0077_r_0_mu_0}\includegraphics[scale=\imagescale]{a_testiu_rho_004_r_0_mu_0}
\includegraphics[scale=\imagescale]{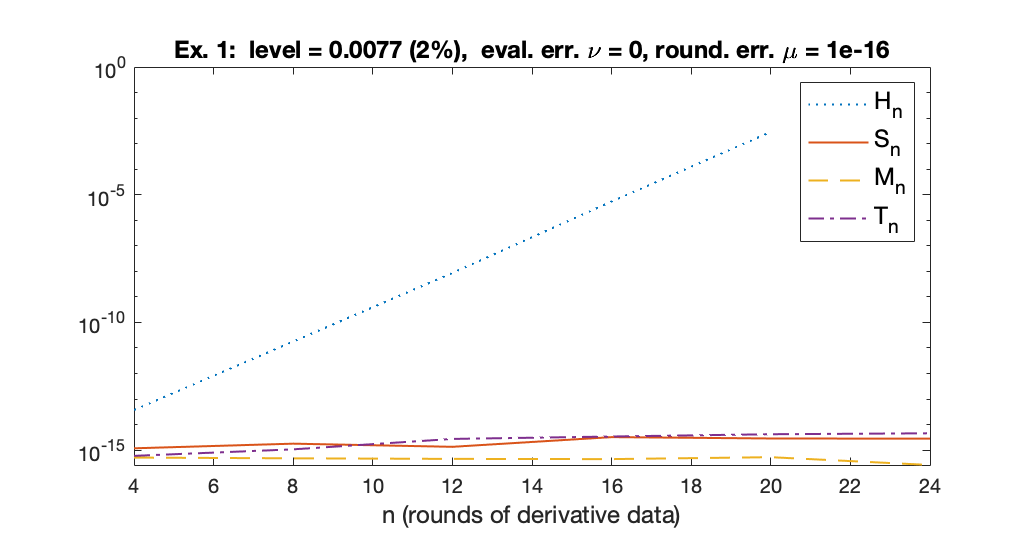}\includegraphics[scale=\imagescale]{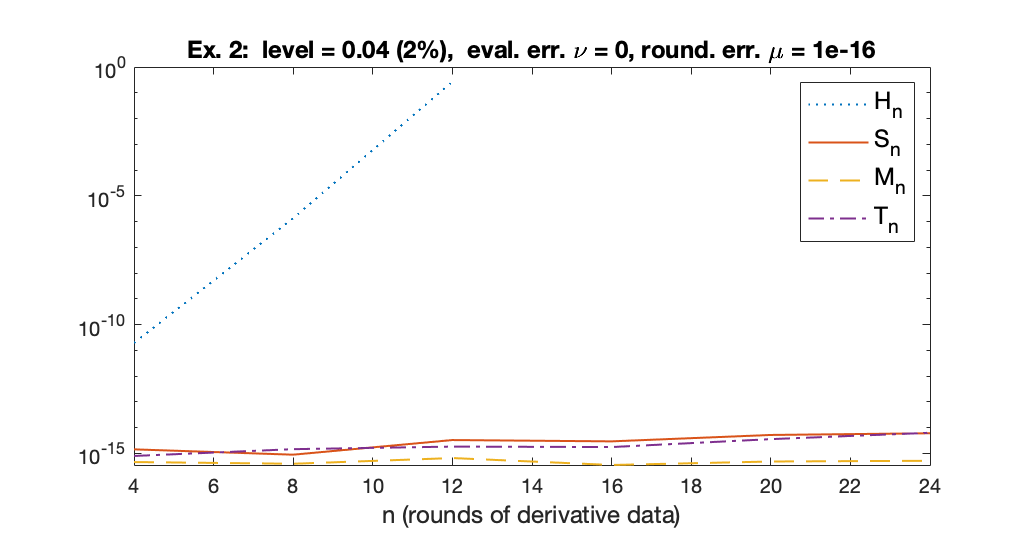}
\includegraphics[scale=\imagescale]{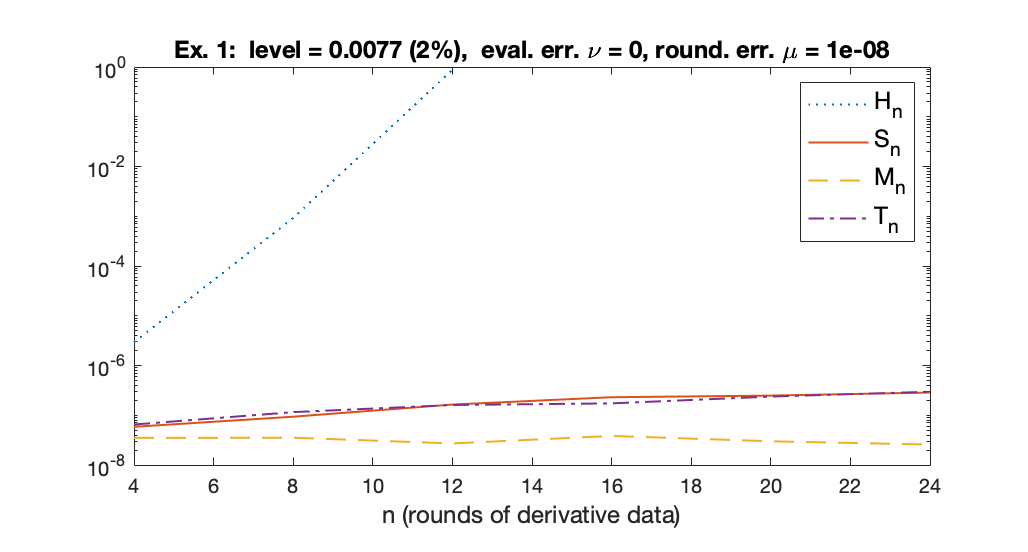}\includegraphics[scale=\imagescale]{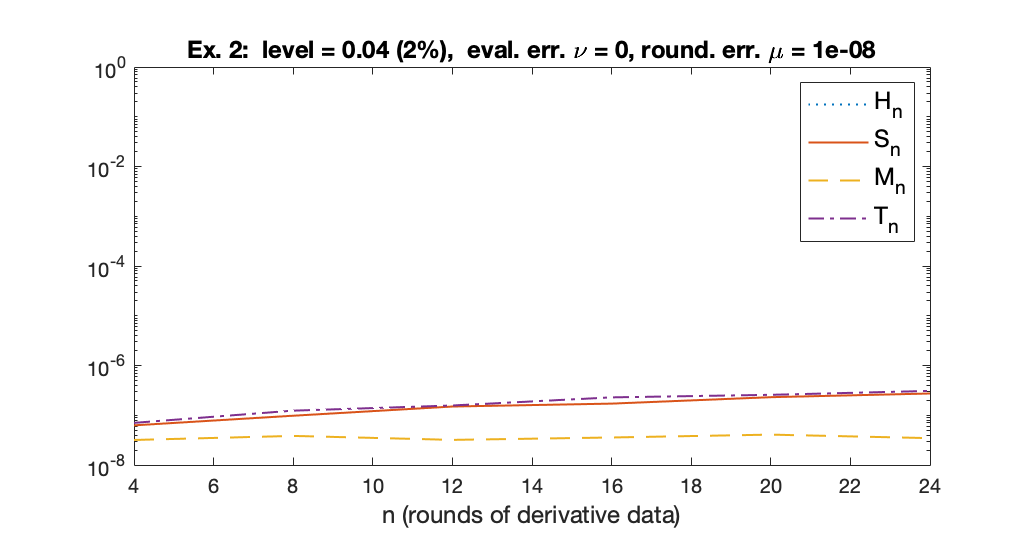}
\includegraphics[scale=\imagescale]{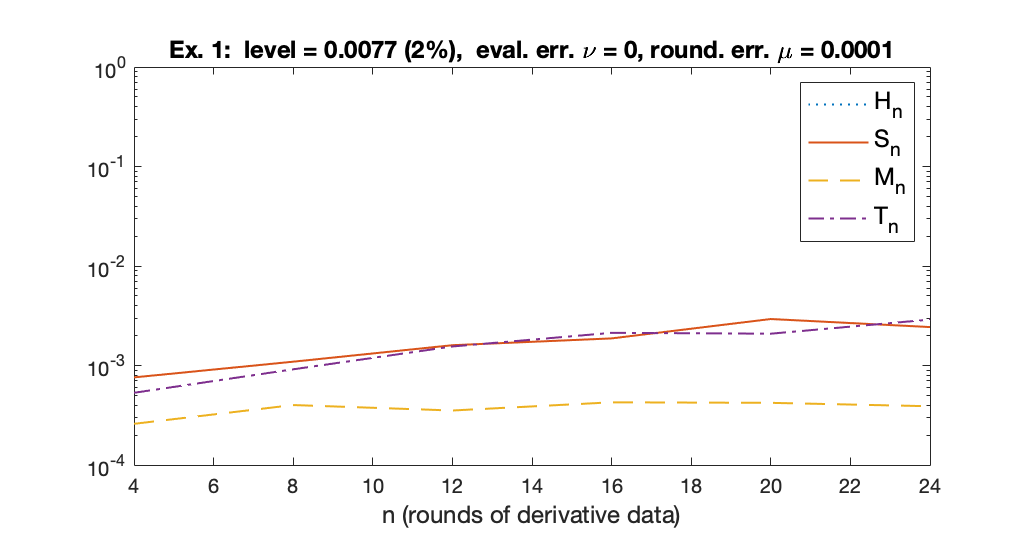}\includegraphics[scale=\imagescale]{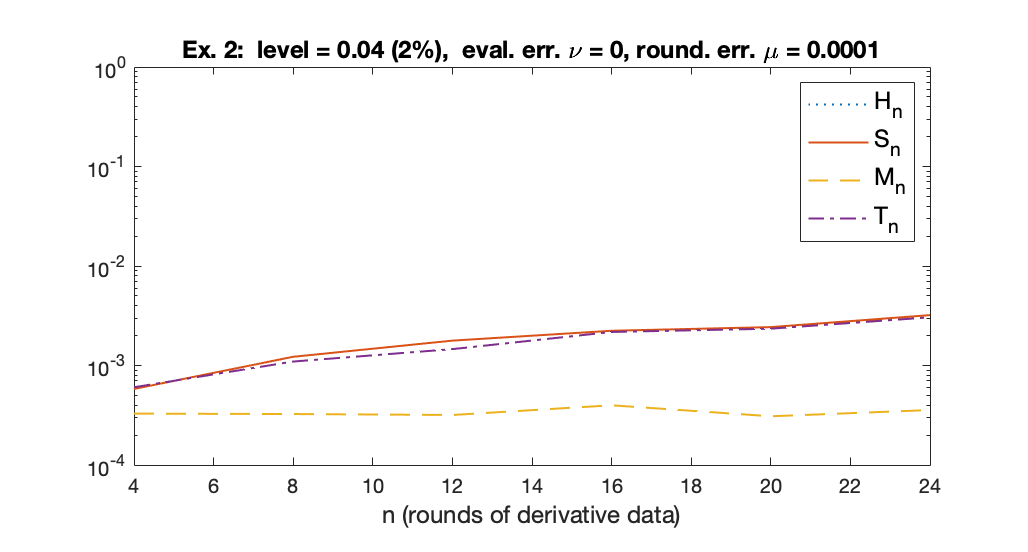}
\end{center}
\caption{(Example 1 on the left and 2 on the right) The effect of rounding error $\mu$ when $\nu=0$ and $\rho=0.02\times \rho_{crit}$.  
Values used for $\mu$ are $0$, $10^{-16}$, $10^{-8}$, and $10^{-4}$ (from top to bottom). The error values for Special (red contiguous line) on the bottom right are $6.54\cdot10^{-4}$ at $n=4$
and $2.9\cdot10^{-3}$ at $n=24$.}
\label{myyn_vaikutus}
\end{figure}

\begin{figure}[h]
\newcommand{\imagescale}{0.17}
\begin{center}
\includegraphics[scale=0.2]{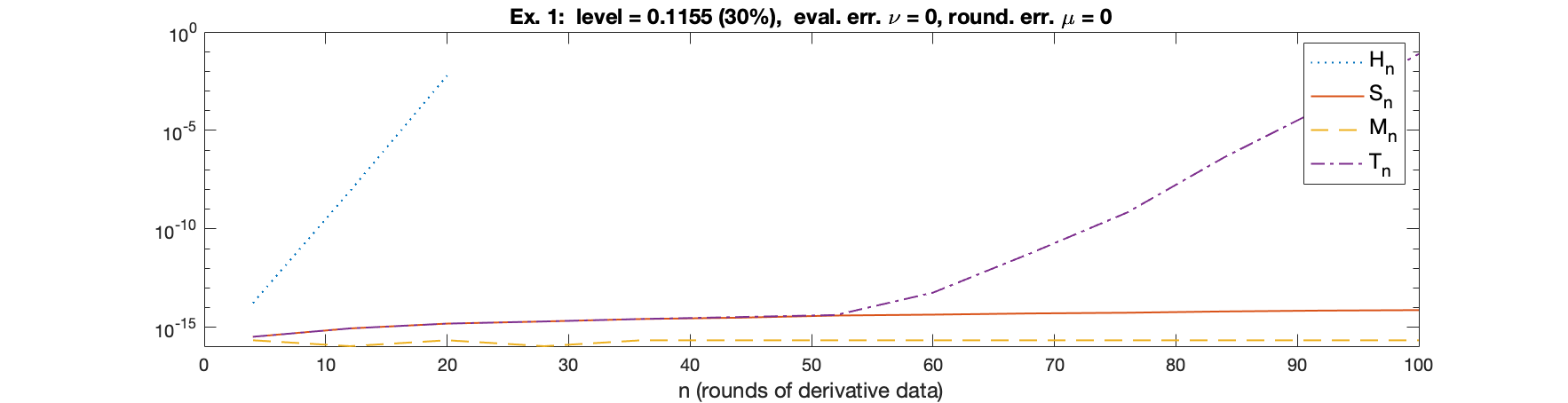}
\includegraphics[scale=0.2]{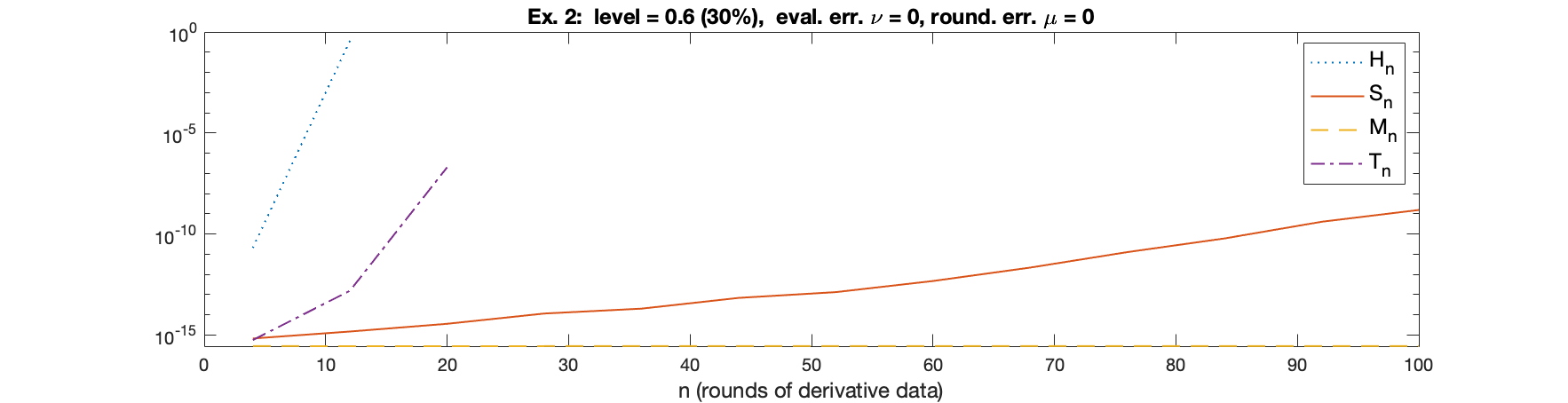}
\end{center}
\caption{(Example 1 (top) and 2 (bottom)) 
The effect of large values for the number of round of derivative data $n$, where $n=4:8:100$. Other variables are kept fairly small, $\rho = 0.30\times\rho_{crit}$, or zero as $\nu = \mu = 0$.}
\label{n_100}
\end{figure}

\begin{figure}[h]
\newcommand{\imagescale}{0.17}
\begin{center}
\includegraphics[scale=0.15]{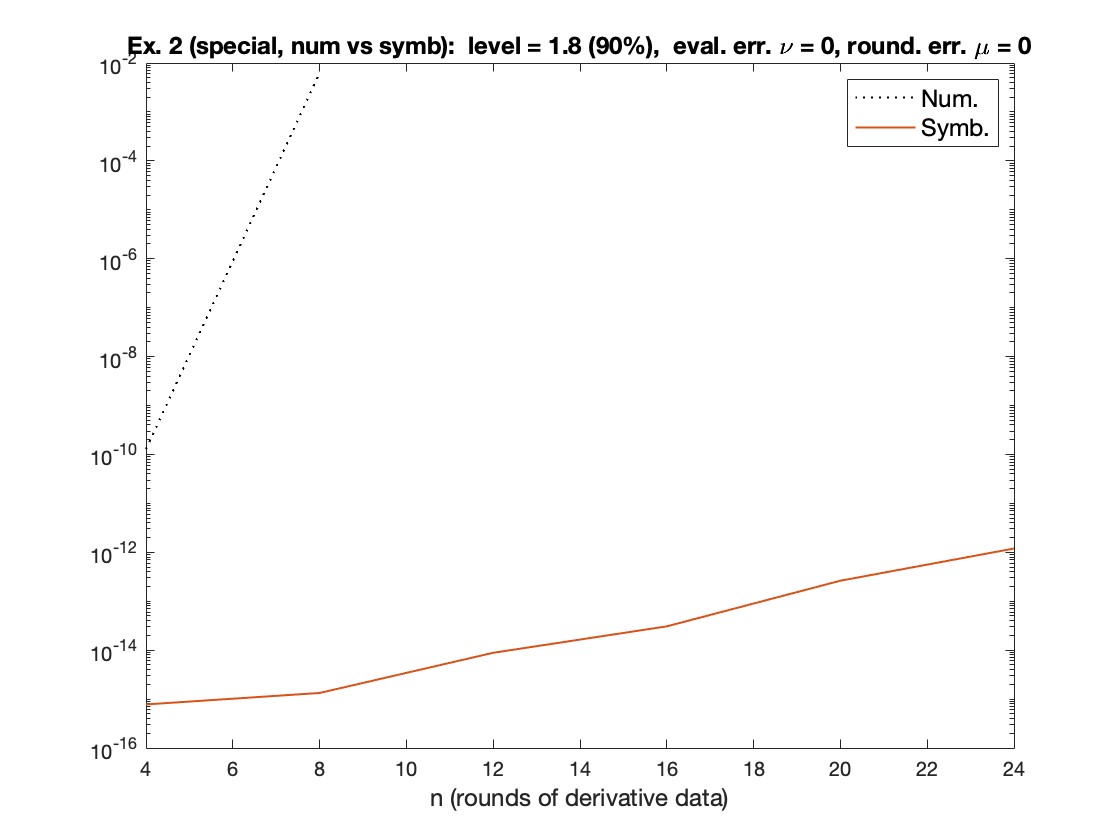}
\includegraphics[scale=0.18]{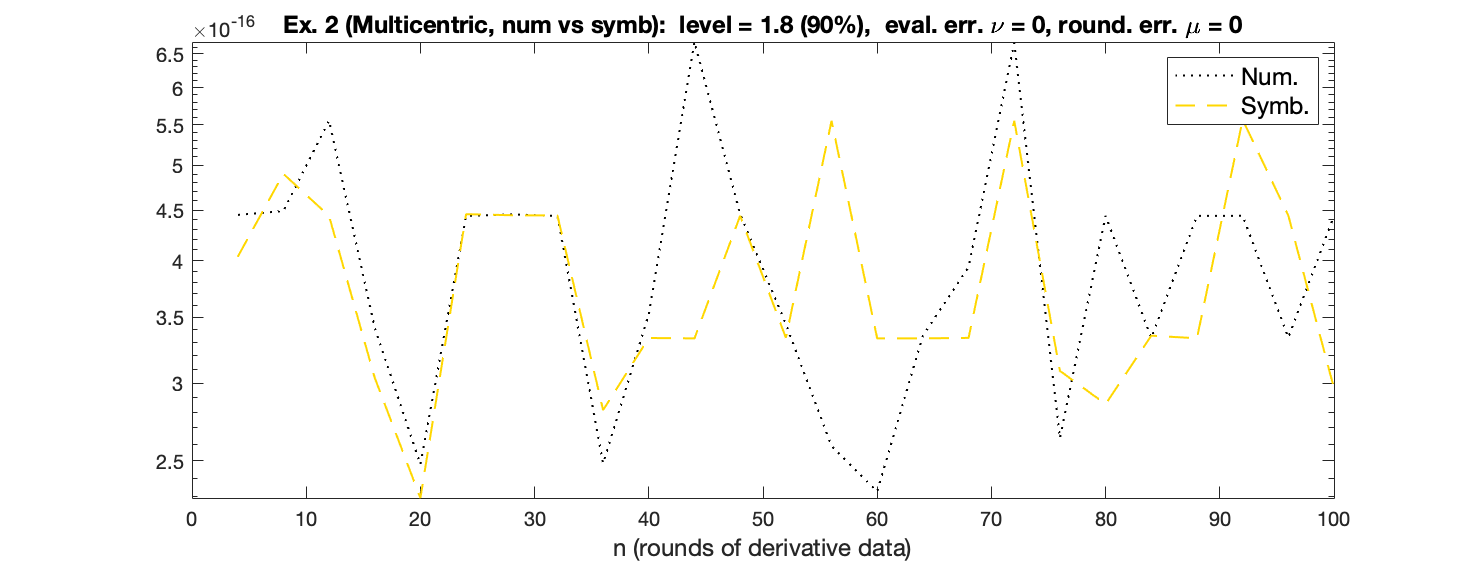}
\end{center}
\caption{From the top graph we see that computing the coefficients numerically introduces drastically more errors 
for the Special algorith in the Example 2 with $\rho=1.80$ $(90\%)$. (For Parallel algorithm the results were even worse than this.) With Multicentric the difference between symbolically and numerically computed coefficients stays withing $10^{-16}$ even for large values $n$ in the bottom graph.}
\label{coeffs_numerically}
\end{figure}

\end{document}